\documentclass[11pt]{article}
\usepackage{amsthm}
\usepackage{latexsym, amssymb, enumerate, amsmath}

\def\rbb{\mathbb{R}}

\title{Anomalous diffusion of distinguished particles in bead-spring networks}
\author{Scott A.~McKinley\thanks{Department of Mathematics, Duke University, Box 90320,
Durham, NC 27701 \texttt{mckinley@math.duke.edu}}}
\date{\today}

\begin{document}

\sloppy

\renewcommand{\theequation}{\arabic{section}.\arabic{equation}}

\def\open#1{\setbox0=\hbox{$#1$}
\baselineskip = 0pt \vbox{\hbox{\hspace*{0.4 \wd0}\tiny
$\circ$}\hbox{$#1$}} \baselineskip = 10pt\!}

\newcounter{assume}

\newtheorem{thm}{Theorem}[section]
\newtheorem{prop}[thm]{Proposition}
\newtheorem{lem}[thm]{Lemma}
\newtheorem{cor}[thm]{Corollary}

\newtheorem{defn}[thm]{Definition}
\newtheorem{notation}[thm]{Notation}
\newtheorem{example}[thm]{Example}
\newtheorem{conj}[thm]{Conjecture}
\newtheorem{prob}[thm]{Problem}
\newtheorem{assumption}[assume]{Assumption}

\newtheorem{remark}[thm]{Remark}

\newcounter{other}
\newcounter{prove}

\newtheorem{theorem}{Theorem}[section]

\def\ep{\epsilon}
\def\ups{\Upsilon}
\def\ddt{\frac{d}{dt}}
\def\drag{\beta}
\def\viscosity{\eta}
\def\f{\varphi}
\def\sech{\mbox{sech}}
\def\diffeff{\sigma}
\def\lengtheff{\mathcal{L}}
\def\rel{\ep}
\def\acf{\sigma}
\def\msd{\sigma}
\def\sou{$\Sigma \mbox{OU }$}
\def\pbb{\mathbb{P}}

\def\schwartz{\mathscr{S}}

\def\Av{\mathbf{A}}
\def\Bv{\mathbf{B}}
\def\Dv{\mathbf{D}}
\def\Fv{\mathbf{F}}
\def\Pv{\mathbf{P}}
\def\Xv{\mathbf{X}}
\def\Wv{\mathbf{W}}
\def\Zv{\mathbf{Z}}

\def\Abf{\mathbf{A}}
\def\Bbf{\mathbf{B}}
\def\Dbf{\mathbf{D}}
\def\Fbf{\mathbf{F}}
\def\Ibf{\mathbf{I}}
\def\Lbf{\mathbf{L}}
\def\Pbf{\mathbf{P}}
\def\Qbf{\mathbf{Q}}
\def\Xbf{\mathbf{X}}
\def\Wbf{\mathbf{W}}
\def\Zbf{\mathbf{Z}}
\def\rbf{\mathbf{r}}
\def\vbf{\mathbf{v}}
\def\xbf{\mathbf{x}}
\def\zbf{\mathbf{z}}
\def\zerobf{\mathbf{0}}

\def\ebb{\mathbb{E}}
\def\nbb{\mathbb{N}}
\def\zbb{\mathbb{Z}}

\def\lambdabf{\boldsymbol{\lambda}}
\def\Lambdabf{\boldsymbol{\Lambda}}

\def\ecal{\mathcal{E}}
\def\gcal{\mathcal{G}}
\def\kcal{\mathcal{K}}
\def\lcal{\mathcal{L}}
\def\ucal{\mathcal{U}}

\def\lscr{\mathscr{L}}

\def\partialt{\frac{\partial}{\partial t}}
\def\partialx{\frac{\partial}{\partial x}}
\def\partialxx{\frac{\partial^2}{\partial x^2}}

\newcommand{\E}[1]{\mathbb{E}{\left[ #1\right]}}
\newcommand{\Ec}[1]{\mathbb{E}^c{\left[ #1\right]}}
\newcommand{\Ex}[2]{\mathbb{E}_{#1}{\left[ #2\right]}}
\newcommand{\p}[1]{\mathbb{P}{\left\{ #1\right\}}}
\newcommand{\Var}[1]{\mathrm{Var}{\left( #1\right)}}
\newcommand{\Varc}[1]{\mathrm{Var}^c{\left( #1\right)}}
\maketitle

\begin{abstract}
We consider the anomalous sub-diffusion of a class of Gaussian processes
that can be expressed in terms of sums of Ornstein-Uhlenbeck
processes. As a generic class of processes, we introduce a single parameter such that for any $\nu \in (0,1)$ the process can be tuned to produce a mean-squared displacement with $\E{x^2(t)} \sim t^\nu$ for large $t$.

The motivation for the specific structure of these sums of OU processes comes from the Rouse chain model from polymer kinetic theory. We generalize the model by studying the general dynamics of individual particles in networks of thermally fluctuating beads connected by Hookean springs. Such a set-up is similar to the study of Kac-Zwanzig heat bath models.  Whereas the existing heat bath literature places its assumptions on the spectrum of the Laplacian matrix associated to the spring connection graph, we study explicit graph structures. In this setting we prove a notion of universality for the Rouse chain's well-known $\E{x^2(t)} \sim t^\frac{1}{2}$ scaling behavior. Subsequently we demonstrate the existence of other anomalous behavior by changing the dimension of
the connection graph or by allowing repulsive forces among the beads.
\end{abstract}

\section{Introduction}

Due to recent and compelling experimental observations using
advanced microscopy \cite{2005-suh,2006-VirtualLung-PNAS,2007-hanes}
there is theoretical interest
\cite{2002-morgado-anom-diff,2004-kupferman-frac-kin,2008-vainstein-scaling-law,2008-santamariaholek,2008-kou-subdiffusion,2009-fricks,2009-jor}
in anomalous diffusion -- stochastic processes whose long-term
mean-squared displacement (MSD) satisfies $\E{x^2(t)} \sim t^\nu$
where $\nu \neq 1$. In each of the cited references, the observed
behavior is sub-diffusive, where $\nu \in (0,1)$. The canonical
example of a sub-diffusive process is fractional Brownian motion
\cite{2008-kou-subdiffusion}, but in this paper, we focus on a model
from polymer kinetic theory and natural generalizations.

The Rouse chain model of a polymer is a series of thermally
fluctuating beads $\{x_n(t)\}_{t \geq 0}$, $n \in \{1, \ldots, N\}$
that interact with nearest neighbors through linear spring forces.
It is a standard observation in the physics literature
\cite{book-doi-edwards,1990-kremer-grest,book-rubinstein-colby} that
while the center-of-mass of the chain is a diffusive process, there
exist positive times $\tau_1$ and $\tau_N$ such that individual
beads roughly exhibit the following MSD profile:
\begin{equation} \label{eq:sou-msd-profile}
\E{x_n^2(t)} \sim \left\{\begin{array}{cl}
t,& t \ll \tau_1\\
t^{\frac{1}{2}},& \tau_1 \ll t \ll \tau_N \\
t,& t \gg \tau_N \\
\end{array}\right.
\end{equation}
The times $\tau_1$ and $\tau_N$ are called the first and last
relaxation times respectively.  Two of the primary projects in the
present paper are to give precise mathematical meaning to a profile
such as \eqref{eq:sou-msd-profile} and to show that the anomalous
exponent $\nu = \frac12$ in the intermediate timescale is determined
by the geometric structure of the graph of connections among the
beads.

The observation that the dynamics of an individual particle in a
network, a so-called \emph{distinguished particle process}, can
exhibit anomalous diffusion is not restricted to polymer kinetics.
In a series of papers
\cite{2002-kupferman-stuart,2004-kupferman-frac-kin,2004-kupferman-fitting-sde}
the authors studied the behavior of a distinguished particle in a
Kac-Zwanzig heat bath, a model used in molecular dynamics theory to
study the force exerted on a particle by a randomly fluctuating
environment. Of present interest are the articles
\cite{2004-kupferman-frac-kin} and \cite{2008-kou-subdiffusion}
wherein the authors showed that in an appropriately constructed
large-$N$ limit, a family of distinguished particle processes can
converge weakly to a sub-diffusive limiting process. In
\cite{2008-kou-subdiffusion}, this process is fractional Brownian
motion, while in \cite{2004-kupferman-frac-kin} the limiting process
is the so-called generalized Langevin equation (see also
\cite{2001-zwanzig-book}) with a power law memory kernel. It is
worth noting that in each of the above cases, the results followed
from assumptions that were placed on the spectrum of the weighted
spring connection graph, rather than directly on its weights and
geometric structure, which is the goal of the present paper.

In \cite{2009-jor}, the authors introduce a common mathematical
framework to address the sub-diffusion seen in these models: a class
of Gaussian processes expressible in terms of a Brownian motion plus
a sum of Ornstein-Uhlenbeck processes,
\begin{equation} \label{eq:sou-defn-x}
x(t) = c_0 B_0(t) + \sum_{k=1}^{N-1} c_k z_k(t)
\end{equation}
where each $z_k$ satisfies the SDE
\begin{equation}
\label{eq:sou-defn-z} d z_k(t) = - \lambda_k z_k(t) + d B_k(t)
\end{equation}
The collection of standard Brownian motions $\{B_{k,N}\}_{k \leq
N-1}$ are assumed to be independent. The positive constants
$\{\lambda_{k}\}_{k \leq N-1}$ will be called the \emph{diffusive
spectrum} and the constants $\{c_{k}\}_{k \leq N-1}$ will be called
the \emph{coefficient family}. The inverses of the elements of the
diffusive spectrum $\tau_k := \lambda_k^{-1}$ are called the
\emph{relaxation times} of the process.  Henceforth we will refer to
processes defined by \eqref{eq:sou-defn-x} and \eqref{eq:sou-defn-z}
as \sou processes.

One can think of an \sou process as a diffusing particle that is
free to explore all of space, but is subject to a sequence of linear
mechanical responses from its environment. The central notion of this paper is that
anomalous diffusion can arise from the structure of timing of this
cascade of responses. One can directly show (see Section
\ref{subsec:distinguished-particles-defn}) that when the beads in a
network interact through linear spring forces, the associated
distinguished particle process is exactly expressible in terms of
equations \eqref{eq:sou-defn-x} and \eqref{eq:sou-defn-z}.  %For a

In \cite{2009-jor}, the authors laid out the general relationship
between the diffusive spectrum $\{\lambda_{k}\}_{k=1}^{N-1}$ and the
intermediate timescale anomalous exponent for \sou processes.
Generalizing the Rouse spectrum  $\lambda_k = \sin^2(k \pi / 2N)$ from the polymer kinetic theory \cite{book-doi-edwards,book-rubinstein-colby} by defining the
diffusive spectrum to be
\begin{equation} \label{eq:zimm-spectrum}
\lambda_{k,N} = \left(\frac{k}{N}\right)^\rho \tau_1^{-1},
\end{equation}
we find that \sou processes can exhibit any desired
anomalous exponent between 0 and 1. Indeed, the full MSD profile of
an \sou process with generalized Rouse spectrum is given by 
\begin{equation} %\label{eq:sou-msd-profile}
\E{x^2(t)} \sim \left\{\begin{array}{cl}
t,& t \ll \tau_1\\
t^{1 - \frac{1}{\rho}},& \tau_0 \ll t \ll \tau_N \\
t,& t \gg \tau_N \\
\end{array}\right.
\end{equation}
for any $\rho > 1$. The actual exponents observed in experiments
vary widely, and so the existence of stochastic processes that
exhibit robust and varied anomalous behavior is appealing to
experimentalists and engineers.  At present, this community still
lacks effective statistical inference tools for these sub-diffusive
processes as well as the ability to conduct simulated experiments
that are not computationally prohibitive.

\subsection{Summary of Results}

In this paper, we seek to give a rigorous interpretation to the
claims made in the companion work \cite{2009-jor}. In Theorem
\ref{thm:main} we give precise meaning to the MSD profile
\eqref{eq:sou-msd-profile}. The key observation, independently noted
in \cite{2008-kou-subdiffusion}, is that for these models, the
largest relaxation time $\tau_N$ tends to infinity with the number
of modes $N$. The family of processes is tight, with a sequence of
autocorrelation (ACF) functions that converge uniformly on compact sets. Therefore,
demonstrating anomalous diffusion for large $t$ in the limiting
process is tantamount to proving it for the intermediate timescale
of the finite $N$ processes.

There is a robust sense in which the anomalous exponent is
determined by the diffusive spectrum rather than the coefficient
family. This fact is vital if one hopes to perform statistical
inference on an \sou process, see \cite{2009-jor} for further
discussion. In Proposition \ref{prop:coeff-subdominant} we show that
setting the coefficients to be i.i.d.~random variables with mild
restrictions, the exponent $\nu$ remains the same.

While the generic \sou structure can support all anomalous
sub-diffusive exponents it turns out that the behavior of
distinguished particle processes is not so varied. It is conjectured
in the physics literature \cite{book-doi-edwards} that the Rouse
model is a universality class that captures the qualitative behavior
of a wide variety of bead-spring networks. The main result of Section \ref{sec:distinguished-particles} is that this is indeed true in some rigorous sense (Theorem \ref{thm:rouse-universal}).  The only exponent seen for a wide class
of models is $\nu = \frac{1}{2}$. However, it is possible to
construct weighted networks that produce different behavior by 1)
changing the dimension of the underlying spring connection graph,
Section \ref{subsec:rouse-higher-d}; and 2) by allowing for
repulsive forces among the beads, Section \ref{subsec:repulsive}.
Ultimately, such modifications cannot account for the wide behavior
seen experimentally.  One will likely need to account for some
combination of hydrodynamic self-interaction \cite{1956-zimm} \cite{1989-ottinger-rabin} and excluded volume effects \cite{book-doi-edwards}, but
rigorous study of these effects without pre-averaging approximations
remains an unsolved problem.

\subsection{Distinguished particles in bead-spring networks}
\label{subsec:distinguished-particles-defn}

We demonstrate the connection between distinguished particle
processes and \sou processes.  Let $\xbf(t) = (x_1(t), x_1(t),
\ldots, x_2(t), \ldots, x_N(t))$ denote the locations in $\rbb^d$ of
a set of particles at time $t \geq 0$. For the sake of simplicity we
take $d=1$, although this is not essential (see Section
\ref{subsec:high-dim}). Following the development of flexible
polymer kinetics \cite{book-doi-edwards,book-rubinstein-colby}, the
particles are subject to random thermal fluctuations while
interacting through a given quadratic configuration potential
$$
\Psi(\xbf) = \frac{1}{2} \sum_{n \neq m} \kappa_{nm} |x_n - x_m|^2.
$$
The set of pairs $\ecal := \{(x_n, x_m) : \kappa_{nm} > 0\}$
constitutes the set of edges of the graph $\gcal$ associated with
network.

Particle dynamics are formally set by a balance of forces through
the Langevin equation,
$$
m \ddot \xbf = - \eta \dot \xbf(t) - \nabla \Psi(\xbf(t)) + \sigma
\dot \Wbf(t)
$$
where $\eta$ is the viscosity of the fluid in which the particles
are immersed, and the strength of the noise $\sigma$ is related to
the viscosity through the fluctuation dissipation theorem: $\sigma
= \sqrt{2 k_B T \eta}$.  The constant $T$ is the temperature of the
fluid and $k_B$ is Boltzmann's constant. For simplicity we will
renormalize the dynamics so that $\eta = 1$.

The common mass $m$ of the particles is considered to be small. In
the companion paper \cite{2009-jor}, we observe that the zero-mass
limit is singular and non-trivial to analyze.  Here we
restrict our attention to the weak (overdamped) zero-mass limit, as described in
\cite{2009-jor}, which amounts to setting $m=0$ and heretofore
taking the stochastic forcing term $\Wbf(t)$ to be an i.i.d vector
of standard Brownian motions.

The force exerted by the configuration potential on the $n$-th bead
is
$$
-\nabla_{x_n}\!\Psi(\xbf) = \sum_{m \ne n} \kappa_{nm} (x_m - x_m)
$$
leading to the linear system of SDEs
\begin{equation} \label{eq:xbf-sde}
d \xbf(t) = \Lbf \xbf(t) dt + \sigma d\Wbf(t)
\end{equation}
where $\Lbf$ is the so-called Laplacian matrix for the spring
connection graph $\gcal$. The Laplacian matrix is sometimes written
$\Lbf = \Abf - \Dbf$ where $\Abf$ is the weighted adjacency matrix
for $\gcal$ and $\Dbf$ is a diagonal matrix whose entries are the
sums of spring constants $D_{nn} = \sum_{m} \kappa_{nm}$.

We note that $\Lbf$ is symmetric and negative definite. As such it
can be diagonalized in the form
\begin{equation} \label{eq:L-diag}
\Lbf = \Qbf \Lambdabf \Qbf^{-1}
\end{equation}
where $\Lambdabf$ is a diagonal matrix with the eigenvalues of
$\Lbf$ as its entries, and $\Qbf$ is an orthogonal matrix with the
eigenvectors of $\Lbf$ as its columns. We make use of a few standard observations. First, 0 is always an
eigenvalue of $\Lbf$ and if $\gcal$ is connected, then 0 has
multiplicity 1.  The eigenvector associated to 0 has the form $(1,
1, \ldots, 1)'$. Second, all non-zero eigenvalues of $\Lbf$ are
strictly negative.  These will be denoted $\{-\lambda_k\}$ with $k = 1, \ldots N-1$.

One may work with the system \eqref{eq:xbf-sde} by taking a discrete
Fourier transform, however we will use the eigendecomposition
\eqref{eq:L-diag} to define the so-called \emph{normal modes}: $\zbf
:= \Qbf^{-1} \xbf$. We readily see that these modes satisfy a
non-interacting system of SDEs,
\begin{equation} \label{eq:zbf-sde}
d \zbf(t) = \Lambdabf \zbf(t) dt + \sigma d \Bbf(t)
\end{equation}
where $\Bbf = \Qbf^{-1} \Wbf$.  Since $\Qbf$ is a orthogonal, the
rows of $\Qbf^{-1}$ form an orthonormal family of vectors and it
follows that $\Bbf$ is a vector of independent standard Brownian
motions.

We observe that the mode $z_0$, associated to the eigenvalue 0, is
simply a standard Brownian motion. Recalling that the form of the
eigenvector associated with the eigenvalue 0 is $(1,1,1, \ldots,
1)'$, it immediately follows that the ``center of mass'' of the
bead-spring network $\bar x(t) := \frac{1}{N} \sum_{n=1}^{N} x_n(t)
= \frac{1}{N} z_0(t)$ is also a standard Brownian motion with
diffusion coefficient $\sigma / \sqrt{N}$.

However, we will see that individual particles are sub-diffusive processes with
an exponent that depends on the details of the network.  We
transform back into real coordinates by multiplying
\eqref{eq:zbf-sde} on the left by $\Qbf$.  This recovers the \sou
process form of each particle:
\begin{equation} \label{eq:distinguished-sum-of-ou-x}
x_n(t) = \frac{\sigma}{\sqrt{N}} B_0(t) + \sum_{k=1}^{N-1}
q_{k+1,n+1} z_k(t)
\end{equation}
where the coefficients $\{q_{kn}\}_{k,n \leq N}$ are the entries of
$\Qbf$ and the $\{z_k\}_{k\leq N-1}$ are defined by the system of
SDEs
\begin{equation} \label{eq:distinguished-sum-of-ou-z}
d z_k(t) = - \lambda_k z_k(t) dt +
\sigma d B_k(t).
\end{equation}

\subsection{Touchstone example: the Rouse chain}
\label{subsec:rouse-intro}

We may now discuss the Rouse chain model in the context of \sou
processes.  The graph $\gcal_R$ associated with this model consists
of edges $x_n \leftrightarrow x_{n+1}$ for all $n = 1, \ldots, N,$
each with spring constant $\kappa$. We also include the edge $x_1
\leftrightarrow x_N$ so that the particles in the system are
exchangeable. This yields the system of SDEs
\begin{equation} \label{eq:rouse-interaction}
d x_n(t) = \kappa [x_{n-1}(t) - x_n(t)] + \kappa [x_{n+1}(t) -
x_n(t)] dt + \sigma dW_n(t),
\end{equation}
which can be summarized by the vector equation
$$
\dot \xbf (t) = \kappa \Lbf \xbf(t) + \sigma d \Wbf(t)
$$
where $\Lbf$ is the tridiagonal matrix
$$
\Lbf = \left(\begin{array}{rrrrrcrrr}
-2 & 1 & 0 & \phantom{-}0 & \phantom{-}0 & \phantom{-}\ldots & \phantom{-}0 & 0 & 1\\
1 & -2 & 1 & 0 & 0 & \phantom{-}\ldots & 0 & 0 & 0 \\
0 & 1 & -2 & 1 & 0 & \phantom{-}\ldots & 0 & 0 & 0 \\
&&&&& \phantom{-}\vdots &&&\\
0 & 0 & 0 & 0 & 0 & \phantom{-}\ldots & 1 & -2 & 1 \\
1 & 0 & 0 & 0 & 0 & \phantom{-}\ldots & 0 & 1 & -2 \\
\end{array} \right)
$$
The eigenvalues of the matrix $\Lbf$ are given by
$$
- \lambda_{k} = - 4 \sin^2\!\left(\frac{k \pi}{N}\right).
$$
for $k = 0, \ldots, N-1$ (see Section \ref{subsec:rouse-universality}). The fact that the eigenvalues are selected from what we shall call a \emph{spectral shape function}
$$
\f(x) = 4\sin^2(\pi x)
$$
is an essential feature to all the models studied in this paper (see
Assumption \ref{as:shape} in Section \ref{sec:sums-of-ous}).

The well known observation
\cite{book-doi-edwards,1990-kremer-grest,book-rubinstein-colby} that
distinguished particles in Rouse chains are sub-diffusive with
exponent $\nu = \frac12$ follows from noting that for small values
of $x$, the shape function $\f(x)$ behaves essentially like $x^2$, a
notion that is generalized by Assumption \ref{as:rho}. Then one can
apply Laplace's method to the MSD to determine the asymptotic
behavior, see proof of Theorem \ref{thm:main}.

\section{Sums of Ornstein-Uhlenbeck processes}
\label{sec:sums-of-ous}

Define the family of processes
\begin{equation} \label{eq:x-defn}
x_N(t) = c_{0,N} B_{0,N} + \sum_{k=1}^{N-1} c_{k,N} z_{k,N}(t)
\end{equation}
where for each $(k,N)$ the Ornstein-Uhlenbeck processes
$\{z_{k,N}\}_{k=1}^{N-1}$ satisfy the SDEs
\begin{equation} \label{eq:z-defn}
d z_{k,N}(t) = - \lambda_{k,N} z_{k,N}(t) dt + d B_{k,N}(t).
\end{equation}
Dependence on $N$ will be suppressed wherever there is no chance of
ambiguity.

Our focus will be on systems, such as the Rouse model, for which the
diffusive spectrum can be analyzed asymptotically in $N$. The
eigenvalues of the Laplacian matrix associated to the Rouse
connection graph $\gcal_R$ converged to the continuous shape function $4
\sin^2(\pi x)$ and we generalize the notion as follows.
\begin{assumption}[Diffusive spectrum shape function]
\label{as:shape} There exists a nonnegative continuous function $\f
\in L_1([0,1])$, that is strictly postive for all $x \in (0,1)$ with
$\f(0) = 0$ such that
$$
\lim_{N \to \infty} \sup_{k \in \{1, \ldots, N-1\}} \left\{\left|
\lambda_{k,N} - \f\Big(\frac{k}{N}\Big) \right| \right\}= 0
$$
\end{assumption}
Continuity of $\f$ along with the specification of the value $\f(0) = 0$ assures that the longest relaxation time $\tau_N$, which is the inverse of the smallest spectral
value, tends to infinity with $N$. As is mentioned in the Discussion
section at the end of this paper, there are natural generalizations
such as subdiffusion in a quadratic potential where this will not be
the case.

We will see that the behavior of the shape function near zero
determines the most important qualitative dynamics of \sou processes.
\begin{assumption}[Spectral parameter $\rho$] \label{as:rho}
The shape function $\f$ has a Frobenius expansion \cite{book-bender-orszag}
$$
\f(x) \sim x^{\rho} \sum_{n=0}^\infty a_n x^{n} \qquad \mbox{ for small } x
$$
for some $\rho > 0$ in the sense that for each fixed $N$,
$$
\lim_{x \to 0} \, x^{-N} \left(\f(t) - x^\rho \sum_{n=0}^N a_n x^{n}\right) =
0.
$$
\end{assumption}

We will see (Proposition \ref{prop:coeff-subdominant}) that the
effect of perturbations to the coefficient family is subdominant to
the shape of the diffusive spectrum. In applications of interest
this happens due to averaging of the coefficients, which we may
characterize in terms of weak convergence of measures. Define for $x \in [0,1]$ the sequence of
\emph{coefficient measures},
\begin{equation} \label{eq:coefficient-measure}
\mu_N(dx) := \sum_{k=0}^{N-1} \delta\!\left(x - \frac{k}{N}\right)
c_{k,N}^2
\end{equation}
where $\delta(x)$ is the Dirac $\delta$-distribution.

\begin{assumption}[Convergence of coefficients]
\label{as:coeff} There exists a nonnegative finite Radon measure
$\mu$ on the interval $[0,1]$ such that $\mu_N \to \mu$ weakly.
\end{assumption}

\bigskip

At this point, we make a note about initial conditions.  The
anomalous behavior in the limiting process is actually the infinite
extension of the transient dynamics of the finite $N$ processes. In
order for the sequence of processes $\{x_N(t)\}$ to be tight, it
must be true that the sequence of initial conditions $\{x_N(0)\}$
must also be tight.  It is natural to choose $x_N(0) = 0$ for all
$N$, but we will further simplify by choosing vanishing initial
condition for each of the OU processes: $z_{k,N}(0) = 0$. We note
that in most relevant cases it is not appropriate to simply choose
each $z_{k,N}$ from its respective stationary distribution. With
such a choice, as $N \to \infty$ the sum of samples from stationary
distributions will not converge.

\subsection{Asymptotic behavior of \sou processes}
\label{subsec:sou-asymptotic}

We seek to relate the structure of the shape function near zero to
the asymptotic anomalous diffusive exponent of an \sou process.  As
mentioned in the Introduction, for any fixed, finite $N$, it is
expected that the MSD profile will have the form $1 / \nu / 1$ over
the three timescale regimes. Before stating a rigorous description
of the dynamics in Theorem \ref{thm:main}, we include some intuitive
discussion.

The short-timescale diffusive regime has both a mathematical and a
physical interpretation.  The mathematical intuition is that
Ornstein-Uhlenbeck processes are locally like Brownian motions.
Therefore for $t \leq \tau_1 := \lambda_{max}^{-1}$, the process
$x_N(t)$ is essentially a finite sum of Brownian motions.
Physically, in the context of distinguished particle dynamics, the
initial diffusive regime results from the fact that for a short
period the beads are able to diffuse independent of the constraints
from the network.

To explain the diffusive behavior on the largest timescale, we first
note that the sum $\sum_{k=1}^{N-1} c_k z_k$ converges to a
stationary distribution which is normal with mean zero and variance
$\sum_{k=1}^{N-1} c_k^2 / (2 \lambda_k)$. The timescale of the
approach to stationarity is dictated by the longest relaxation time
$\tau_N = \lambda_{min}^{-1}$.  For $t > \tau_N$, the process
$x_N(t)$ is a Brownian motion plus a stationary correction and so
the MSD must satisfy $\lim_{t \to \infty} \E{x^2_N(t)}/t = c_0^2$.

We cannot analyze the intermediate regime exactly, but suppose that
Assumption \ref{as:rho} holds for $\rho > 0$.  Then $\lambda_{min}$
will be roughly $(k/N)^\rho$, which means that the longest
relaxation time $\tau_N = N^\rho$ approaches infinity as $N$
increases. As a result, the anomalous stage of the diffusion is
increasingly prolonged, and is infinite in extent in the large $N$
limit. In Theorem \ref{thm:main} we show that this limit exists and
in the course of the proof demonstrate that the MSD of the $x_N(t)$
processes converge uniformly on compact sets.  By performing an
asymptotic analysis on the limiting process we discover the
anomalous exponent and from the uniform convergence of the MSDs, we
see that this is indeed the anomalous exponent seen in the
intermediate phase of the finite $N$ processes.

\bigskip

We are ready to state the main theorem for \sou processes. For the large-$t$ asymptotic statements, we say that $f(t) \sim t^\nu$ if $\lim_{t \to \infty} f(t) / t^\nu = C$ for some nonnegative constant $C$.

\begin{theorem} \label{thm:main}
Let the sets of processes $\{x_N(t)\}$ and $\{z_{k,N}(t)\}$ be
defined by \eqref{eq:x-defn} and \eqref{eq:z-defn}, respectively.
Take $z_{k,N}(0) = 0$ for all $k$ and $N$. Suppose that the
diffusive spectrum $\{\lambda_{k,N}\}$ converges to a shape function
$\f$ in the sense of Assumption \ref{as:shape}. Furthermore, suppose
the coefficients $\{c_{k,N}\}$ satisfy Assumption \ref{as:coeff}
with limiting weight measure $\mu$.

Then the family $\{x_N(t)\}$ converges in distribution as $N \to
\infty$ to a mean zero Gaussian process $x(t)$ defined by its
auto-correlation function
\begin{equation} \label{eq:acf}
\E{x(t) x(s)} = \int_0^1 \frac{e^{-\f(x)|t-s|}}{2 \f(x)} \!\left(1 -
e^{-2 \f(x) (t \wedge s)}\right) \mu(dx).
\end{equation}
If the shape function $\f$ furthermore satisfies Assumption
\ref{as:rho} with spectral parameter $\rho > 0$ and the limiting
weight measure $\mu$ is Lebesgue measure, then
asymptotically, the limiting MSD function $\msd(t) := \E{x^2(t)}$
satisfies
$$
\msd(t) \sim t, \qquad t \mbox{ \emph{near zero.}}
$$
and
\begin{equation} \label{eq:msd-large-t}
\msd(t) \sim \left\{ \begin{array}{cc} t^{1 - \frac{1}{\rho}} &
\rho >  1 \\
\ln t & \rho = 1 \\
1 & 0 < \rho \leq 1 \end{array} \right. \qquad t \mbox{
\emph{large.}}
\end{equation}
\end{theorem}

\begin{remark}
It is important to note that the limits with respect to $N$ and $t$
implicit in \eqref{eq:msd-large-t} are not interchangeable.  For any
finite $N$, $\E{x_N^2(t)} \sim t$ for large $t$ because the Brownian
term eventually dominates the dynamics.  In the absence of the
Brownian term ($c_{0,N} = 0$), the process $x_N(t)$ is positive recurrent.
\end{remark}

\begin{proof}
Convergence in distribution follows from establishing two standard
facts \cite{book-revuz-yor}: convergence of the finite-dimensional
distributions, and tightness in the space $C([0,T])$ of the family of processess $\{x_N\}$ for any $T > 0$. Since each of the processes
in this sequence is Gaussian, convergence of finite-dimensional
distributions follows from pointwise convergence of the
ACFs, $\acf_N(t,s) := \E{x_N(t)x_N(s)}$.
In order to establish tightness we will use the
Kolmogorov criterion \eqref{eq:tightness-condition}. Subsequently,
the asymptotic analysis reduces to an application of Laplace's
method \cite{book-bender-orszag}.

\smallskip

\noindent \emph{Convergence of the finite-dimensional
distributions}: We compute the ACF for $x_N$.  The explicit solution of the respective OU processes is given by
$$
z_k(t) = e^{-\alpha_k t} z_k(0) + \int_{0}^t e^{-\lambda_k (t-t')}
dW_k(t')
$$
where we have suppressed the dependence of the coefficients
$\{c_k\}$ and diffusive spectrum $\{\lambda_k\}$ on $N$. Since the
modes are assumed to be independent with vanishing initial
conditions, we see that for $s,t > 0$. 
$$
\E{z_k(t) z_j(s)} = \delta_{kj} \frac{1}{2 \lambda_{k}} e^{-
\lambda_k |t - s|} \left(1 - e^{- 2 \lambda_k (t \wedge s)}\right) .
$$
where $\delta_{kj}$ is the Kronecker delta-function. Observing that
cross-terms disappear and including the leading term
$\E{B_0(t) B_0(s)} = t\wedge s$, yields
\begin{equation} \label{eq:acf-N}
\acf_N(t,s) = c_0^2 (t \wedge s) + \sum_{k=1}^{N-1} \frac{c_k^2}{2
\lambda_{k}} e^{-\lambda_k |t - s|} \left(1 - e^{- 2 \lambda_k (t
\wedge s)}\right).
\end{equation}
In light of the assumption that $\f(0) = 0$, the above can be
rewritten in terms of the coefficient measures defined in Assumption \ref{as:coeff},
$$
\sigma_N(t,s) = \int_0^1 \frac{e^{-\f(x)|t-s|}}{2 \f(x)} \!\left(1 -
e^{-2 \f(x) (t \wedge s)}\right) \mu_N(dx).
$$
Note that the integrand is continuous for all $x \in (0,1]$ and can
be extended analytically to include $x=0$ for each choice of $t$ and $s$.
The integrand is bounded above by $t \wedge s$ and therefore the
weak convergence of the measures $\mu_N$ implies the limiting
expression \eqref{eq:acf}.

\smallskip

\noindent \emph{Tightness}: As mentioned, tightness of the family of
processes $\{x_N\}$ in $C([0,T])$ is implied by the Kolmogorov
criterion: given $T> 0$, there exists an $N_0 \in \nbb$ and strictly
positive constants $\alpha$, $\beta$ and $C$ such that
\begin{equation} \label{eq:tightness-condition}
\sup_{N \geq N_0} \E{|x_N(t) - x_N(s)|^\alpha} \leq C
|t-s|^{1+\beta}
\end{equation}
for all $s,t \in [0,T]$

From \eqref{eq:acf-N}, we compute
\begin{align*} \E{(z_k(t) - z_k(s))(z_j(t) - z_j(s))} &=
\delta_{jk} \frac{1}{2\lambda_k} (2 - e^{-2 \lambda_k (t \wedge s})
(1- e^{-2\lambda_k |t-s|}),
\end{align*}
while $\E{(B_0(t) - B_0(s))^2} = |t-s|$. Cross-terms vanish and we
find
\begin{align*}
\E{(x_N(t) - x_N(s))^2} &= c_0^2 |t-s| + \sum_{k=1}^{N-1}
\frac{c_k^2}{2 \lambda_k} (2 - e^{-2 \lambda_k (t \wedge s)}) (1 -
e^{-2\lambda_k|t-s|}) \\
&\leq c_0^2 |t-s| + \sum_{k=1}^{N-1} \frac{c_k^2}{
\lambda_k} (1 - e^{-2\lambda_k|t-s|}) \\
&\leq \left(2\sum_{k=0}^{N-1} c_k^2\right) |t-s|
\end{align*}
In the last line we applied the naive estimate $(1 - e^{\lambda t})
\leq \lambda t$ to each term of the sum.  The sum appearing in the
last line is exactly $\sum_{k=0}^{N-1} \delta(x -
\frac{k}{N})\mathbf{1}_{[0,1]}(x) c_k^2 = \mu_N([0,1])$. By Assumption \ref{as:coeff},
$\mu_N \to \mu$ weakly and by an equivalent statement we have
$$
\limsup_{N \to \infty} \mu_N([0,1]) \leq \mu([0,1]).
$$
As such, there exists an $N_0$ such that for all $N \geq N_0$,
$\mu_N([0,1]) \leq 1 + \mu([0,1])$.   Therefore, for all $N \geq
N_0$,
$$
\E{(x_N(t) - x_N(s))^2} \leq (1 + \mu([0,1])) |t-s|.
$$
Finally, noting that $x_N(t) - x_N(s)$ is Gaussian, we see that
\begin{equation} \label{eq:x-N-increment-fourth-moment}
\sup_{N \geq N_0} \E{(x_N(t) - x_N(s))^4} \leq 3 (1 + \mu([0,1]))^2
|t-s|^2,
\end{equation}
which confirms \eqref{eq:tightness-condition}.

\smallskip

\noindent \emph{Asymptotic analysis}: We now consider the large-$t$
asymptotic behavior of the limiting MSD function $\sigma(t) :=
\E{x^2(t)}$ in the presence of Assumption \ref{as:rho} with shape parameter $\rho$. First we observe that for a given constant $\lambda$
$$
\lambda^{-1} (1 - e^{-\lambda t}) = \int_0^t e^{-\lambda s} ds.
$$
Applying this identity to the integrand in \eqref{eq:acf} and
subsequently using Fubini's Theorem to interchange the integrals
yields
$$
\sigma(t) = \int_0^t \int_0^1 e^{-2 \f(x) s} \mu(dx) ds.
$$
Recalling our assumption that $\mu$ is simply Lebesgue measure, we
write
$$
\sigma(t) = \int_0^t \Phi(s) ds
$$
where $\Phi(s)$ is the Laplace integral $\Phi(s) := \int_0^1 e^{-2
\f(x) s} dx$.

Because $\f(x)$ is continuous (by Assumption \ref{as:shape}) and
therefore bounded, it follows that $\Phi(s)$ is also continuous. By
the Fundamental Theorem of Calculus,
$$
\lim_{t \to 0} \frac{1}{t} \int_0^t \Phi(s) ds = \Phi(0) = 1.
$$
This limit is finite and nonzero, which directly implies that
near zero $\msd(t) \sim t$.

In order to characterize large-$t$ behavior, we note that the minimum value of
$\f(x)$ is assumed to be at $x=0$ and therefore the only significant
contribution to the large $s$ asymptotics will be in a small neighborhood
near zero. Following \cite{book-bender-orszag}, for example, we have
$$
\int_0^1 e^{-\f(x) s} dx \sim \int_0^\infty e^{- x^\rho s} dx =
s^{-\frac{1}{\rho}} \frac{1}{\rho} \Gamma\big(\frac{1}{\rho}\big)
$$
where in the last equality we applied the substitution $y = x^\rho
s$ and $\Gamma$ is the Gamma function $\Gamma(z) := \int_0^\infty
y^{z-1} e^{-z} dz$. Integrating this asymptotic expression while
minding the various ranges of values of $\rho$ yields
\eqref{eq:msd-large-t}.
\end{proof}

\bigskip

%We add two interesting notes to the preceding theorem. 
By generalizing \eqref{eq:x-N-increment-fourth-moment} to higher and
higher moments, one can show that the limiting process $x(t)$ is
$\alpha$-H\"older continuous for any $\alpha \in (0,1/2)$.  This
reinforces the notion that the limiting process is locally like
Brownian motion, but asymptotically like fractional Brownian motion.

\subsubsection{Robustness of the anomalous exponent with respect to
perturbing the coefficients}

As noted in \cite{2009-jor}, if one were interested in conducting
statistical inference on a set of data, trying to fit to an \sou
process, it would be a prohibitive task to fit the tens of thousands
of coefficients.  And Assumption \ref{as:coeff} and the calculation
for the ACF of distinguished particles processes (see proof of
Theorem \ref{thm:distinguished}) can leave the false impression
that the ability to calculate the anomalous exponent $\nu$ is
restricted to special, delicately balanced coefficient families.  In
fact, the result is more robust than this.  We demonstrate this	
phenomenon with the following proposition which presumes random
coefficients.

In what follows there are two senses of averaging.  There is the probability space we use to select the coefficient family and for a fixed coefficient family, there is the probability space associated with the family of Brownian motions that drive the dynamics.  We will use $\mathbb{E}^c$ and $\mbox{Var}^c$ to denote the average and variance with respect to the coefficient probability space. and will continue to use $\ebb$ and $\mbox{Var}$ to denote averages and variances with respect to instances of the Brownian motions.	

\begin{prop} \label{prop:coeff-subdominant}
Suppose the triangular array of coefficients $\{c_{k,N}\}_{k\leq N}$
are independent random variables of the form $c_{k,N} = c /
\sqrt{N}$ where $c$ is a random variable with a finite fourth
moment.  Then the associated coefficient measures
$\{\mu_N\}_{N=1}^\infty$ converge weakly to $\E{c^2}$ times Lebesgue
measure almost surely.

Furthermore, the conclusions of Theorem \ref{thm:main} apply to the
limiting process, $x(t)$.
\end{prop}
\begin{proof}
The result follows from the Strong Law of Large Numbers.

Let $f:[0,1] \to \rbb$ be a continuous function on $(0,1)$ and
denote $m_n := \ebb^c[c^n]$, for $n = \{1, 2, 3, 4\}$. We also recall
the definition of the coefficient measures $\mu_N$ from Assumption
\ref{as:coeff}. Then the random variables
$$
I_N := \int_0^1 f(x) \mu_N(dx) = \sum_{k=1}^{N-1}
f\!\left(\frac{k}{N}\right) c_{k,N}^2
$$
have mean $\ebb^c[I_N] = m_2 \sum_{k=0}^{N-1}
f\!\left(\frac{k}{N}\right) \frac{1}{N}$ and variance
\begin{align*}
\Varc{I_N} &= \Ec{\left(\sum_{k=1}^{N-1} f\!\left(\frac{k}{N}\right)
\left(c_{k,N}^2 - \frac{m_2}{N}\right) \right)^2} \\
&= \sum_{k=1}^{N-1} f^2\!\left(\frac{k}{N}\right) \Ec{\left(c_{k,N}^2
- \frac{m_2}{N}\right)^2}  \\
&= \sum_{k=1}^{N-1} f^2\!\left(\frac{k}{N}\right) (m_4 - m_2^2
)\frac{1}{N^2}
\end{align*}
For the second equality, we note that cross-terms of the sum vanish
due to independence of the coefficients.  It remains to recognize
the Riemann approximation $\sum_{k=0}^{N-1} f^2(k/N) 1/N = \int_0^1
f^2(x) dx + O(1/N)$, which implies that
$$
\Varc{I_N} = (m_4 - m_2^2) \frac{1}{N} \left(\int_0^1 f^2(x) dx +
O(1/N)\right),
$$
which tends to 0 as $N \to \infty$.  Therefore the sequence of
random variables $I_N$ converges almost surely and we conclude that,
in the language of Assumption \ref{as:coeff}, the coefficient
measures $\mu_N$ converge to $m_2$ times Lebesgue measure almost
surely.
\end{proof}

\section{Anomalous diffusion of distinguished particles in bead-spring networks}
\label{sec:distinguished-particles}

\subsection{General framework and main theorem}

When considering the large $N$ scaling limit of bead-spring systems, there are two distinct constructions. In the polymer physics community \cite{book-doi-edwards} \cite{book-rubinstein-colby}, it is typical to simply add a new bead to the bead-spring loop while keeping the spring constants the same.  In mathematical developments, (see \cite{1983-funaki} for example), it is typical to also increase the spring forces while rescaling magnitude of the noise in order to develop a continuum limit of the full bead-spring system.  This is the so-called ``random string'' model.  

While the distinguished particle limit exists in both cases, there is a marked qualitative difference in the behavior of the limiting process.  In the physics development, the limiting process is of the type described in the preceding section: locally like a Brownian motion, but globally sub-diffusive.  In contrast, the limiting distinguished particle process in the ``random string'' development is sub-diffusive on the shortest time scales and approaches a stationary distribution.  As is further discussed in Section \ref{subsec:spde}, the effect of ``bringing the anomalous diffusion to the local scale'' is that the resulting process is locally rougher than Brownian motion, having infinite quadratic variation, but finite quartic variation \cite{2007-swanson}. Our focus is not on this development however, because we are interesting in processes which exhibit anomalous diffusion over arbitrarily large time scales.  One may be concerned that the full chain does not have a limit in the polymer physics construction, but such a limit is not our goal.  Rather, we proceed with the knowledge that there are a finite but large number of beads in the relevant physical systems and although there is no convergence of the full chain structure, there is \emph{convergence of the effect of the chain} on the distinguished particles.

\bigskip

Convergence of diffusive spectrum is sufficient for convergence in distribution of a family of \sou processes.  We now argue the same principle holds for
distinguished particle processes.  The central insight is that the diffusive spectrum in this setting is given by the eigenvalues of the Laplacian matrix associated with the weighted connection graph $\gcal$.  A well-established characterization of stability of a family of graphs is convergence of the eigenvalues to some shape function, as seen for the Rouse chain model in Section \ref{subsec:rouse-intro}.

The only technical detail is that the coefficient
structure of the distinguished particle processes defined in Section
\ref{subsec:distinguished-particles-defn} by Equations
\eqref{eq:distinguished-sum-of-ou-x} and
\eqref{eq:distinguished-sum-of-ou-z} may not precisely satisfy
Assumption \ref{as:coeff} on the coefficient family.  However, after
computing the ACF of the distinguished particle process, we see that
it is equivalent in law to an \emph{effective} \sou process to which
Theorem \ref{thm:main} does apply.

\begin{theorem} \label{thm:distinguished}
Let $\{\gcal_N\}_{N\in\nbb}$ be a sequence of graphs, each having
$N$ vertices respectively, with edge weight sets $\{\ecal_N\}_{N\in
\nbb}$ such that the triangular family $\{\lambda_{k,N}\}_{k \leq
N}$ of eigenvalues of the associated Laplacian matrices $\Lbf_N$
satisfy Assumption \ref{as:shape}. Furthermore, suppose that the
graphs are constructed in such a way that the individual particle
processes are exchangeable.

Then the conclusions of Theorem \ref{thm:main} hold for the family
of processes $\{x_{0,N}(t)\}_{N \in \nbb}$ defined by
\eqref{eq:x-defn} and \eqref{eq:z-defn}.
\end{theorem}

\begin{proof}
Following the notation of Section
\ref{subsec:distinguished-particles-defn} and suppressing dependence
on $N$, the path of the $n$-th particle in the system is given by
$$
x_n(t) = \frac{1}{\sqrt{N}} B_0(t) + \sum_{k=0}^{N-1} q_{n+1,k+1}
z_{k}
$$
where the family of OU-processes $\{z_{k,N}\}_{k=1}^{N-1}$ are
defined by equation \eqref{eq:z-defn}. The coefficients
$\{q_{nk}\}_{k=1}^{N}$ are the $n$-th row of entries of the matrix
$\Qbf$, which we recall is the orthogonal matrix whose columns are
the normalized eigenvectors of the Laplacian matrix $\Lbf_N$.

Because the particles are exchangeable,
\begin{equation} \label{eq:x-exchangeable}
\E{x_n(t)x_n(s)} = \frac{1}{N} \E{\xbf(t)\cdot\xbf(s)}
\end{equation}
where $\xbf(t) = (x_1(t), \ldots, x_N(t))'$ denotes the full vector
of all $N$ particles in the system. Its dynamics are defined by the
vector SDE \eqref{eq:xbf-sde} and the exact solution is given by
Duhamel's formula
$$
\xbf(t) = \sigma e^{\Lbf t} \xbf(0) + \sigma \int_0^t e^{\Lbf(t-s)}
d\Wbf(s)
$$
Again, we recall the assumption that $\xbf(0) = \zerobf$.

In order to calculate the autocorrelation $\E{\xbf(t)\cdot\xbf(s)}$,
we observe:
\begin{align*}
& \ebb\left(\int_0^t e^{\Lbf(t-t')} d \Wbf(t')\right)
\cdot \left(\int_0^s e^{\Lbf(s-s')} d \Wbf(s')\right) \\
& \qquad \qquad = \ebb{ \sum_{k=1}^N \left(\int_0^t e^{\Lbf(t-t')} d
\Wbf(t')\right)_k \left(\int_0^s e^{\Lbf(s-s')} d \Wbf(s')\right)_k} \\
& \qquad \qquad = \ebb \sum_{k=1}^N \left(\sum_{i=1}^N \int_0^t
\left(e^{\Lbf(t-t')}\right)_{ki} d W_i(t')\right) \left(
\sum_{j=0}^N \int_0^s \left(e^{\Lbf(s-s')}\right)_{kj}
d W_j(s')\right) \\
& \qquad \qquad = \sum_{k=1}^N \sum_{i=1}^N \sum_{j=1}^N \delta_{ij}
\int_0^{t \wedge s} \left(e^{\Lbf(t-r)}\right)_{ki}
\left(e^{\Lbf(s-r)}\right)_{kj} dr \\
& \qquad \qquad = \int_0^{t \wedge s}
\sum_{k=1}^N \sum_{j=1}^N \left(e^{\Lbf(t+s-2r)}\right)_{kj} dr \\
& \qquad \qquad = \int_0^{t \wedge s}
\|e^{\frac{1}{2}\Lbf(t+s-2r)}\|_F^2 dr.
\end{align*}
In the last line we have used the Frobenius norm: $ \| A \|_F :=
\sum_{k=1}^N \sum_{j=1}^N a_{kj}^2 = \sum_{k} \lambda_k^2 $ where
$\{\lambda_k\}$ is the set of eigenvalues of $A$.

To complete the calculation above, we note that $e^{\Lbf t}$ is
similar to $e^{\Lambdabf t}$, where $\Lambdabf$ is the diagonal
matrix from \eqref{eq:L-diag} whose entries are the eigenvalues of
$\Lbf$.  Therefore the eigenvalues of $e^{\Lbf t}$ are exactly the
entries $e^{\Lambdabf t}$, namely the set $\{e^{-\lambda_k
t}\}_{k=0}^{N-1}$. Imposing the assumption that $\xbf(0) = \zerobf$,
this implies
\begin{equation} \label{eq:xbf-acf}
\E{\xbf(t)\cdot\xbf(s)} = \sigma^2\int_0^{t \wedge s} \sum_{k=0}^N
e^{- \lambda_k(t+s-2r)} dr.
\end{equation}

Rearranging terms, we see that for each $n$, the distinguished
particle $x_n$ is a mean-zero Gaussian process with ACF of the form
found in Equation \ref{eq:acf-N} with the coefficients identically
set to $c_{k,N} = \sigma / \sqrt{N}$. In this way, we see that the
distinguished particle processes indexed by $N$ are equivalent in
law to a family effective \sou analogues $\{\tilde x_N(t)\}$ defined
by
$$
\tilde x_N(t) = \frac{\sigma}{\sqrt{N+1}} \left(B_0(t) +
\sum_{k=1}^N \tilde z_{k,N}(t)\right)
$$
where
$$
d \tilde z_{k,N}(t) = - \lambda_{k,N} \tilde z_{k,N}(t) dt + d
B_{k,N}(t).
$$
The coefficient measures $\mu_N$, defined by
\eqref{eq:coefficient-measure}, converge weakly to Lebesgue measure
and the asymptotic conclusions of Theorem \ref{thm:main} apply to
$\tilde x_N$ directly and therefore to the distinguished particle
process by corollary.
\end{proof}

\subsubsection{Examples}

We see from the argument above that, after collecting terms appropriately, the coefficient family $\{c_{k,N}\}$ of the effective distinguished particle process
$\tilde x_N(t)$ is given by the respective multiplicities of the
eigenvalue family $\{\lambda_{k,N}\}$. Whereas the coefficient
measures in the Rouse model and its generalizations (Section
\ref{sec:rouse}) will all converge to Lebesgue measure, we take a
moment discuss two examples where this is not the case.

Let $\kcal_N$ denote a complete graph on $N$ vertices.  Then there
are only two eigenvalues: 0, which has multiplicity 1, and $N /
(N-1)$ which has multiplicity $N-1$ \cite{book-chung}. The distinguished
particle process associated to each $\kcal_N$ is equivalent to an
effective \sou process
$$
\tilde x_N(t) = \frac{1}{\sqrt{N}} B_0(t) + \frac{\sqrt{N}}{N-1}
\tilde z_N(t)
$$
where
$$
d \tilde z_N(t) = - \frac{N}{N-1} \tilde z_N(t) + d B_N(t).
$$
The coefficient measures converge to the Dirac-$\delta$ distribution
centered at $x = 1$.  Recalling $x(0) = 0$, we see that the limiting MSD is
given by $\E{x^2(t)} = 1 - e^{-t}$, i.e. a one-dimensional OU
process.

The same asymptotic behavior is observed from a system with a
non-trivial coefficient family structure. An $N$-hypercube on $2^N$
vertices has eigenvalues of the form $\frac{2k}{N}$, with respective
multiplicities $\binom{N}{k}$ \cite{book-chung}. By the de Moivre-Laplace Theorem we
have the following large-$N$ characterization of the coefficients,
$$
c_{k,N} = \binom{N}{k} \frac{1}{2^N} \approx \sqrt{\frac{2}{\pi N}}
e^{-2 N \left(\frac{k}{N} - \frac{1}{2}\right)^2}.
$$
Rewriting the associated coefficients reveals a sequence of
approximate Dirac-$\delta$ functions centered at $x = \frac{1}{2}$,
$$
\mu_N(dx) \approx \sqrt{\frac{2}{\pi}} \sum_{k=1}^{N} \delta^N(x)
\frac{1}{N}
$$
where $\delta^N\!(x) = \sqrt{N} e^{-2 N \left(\frac{k}{N} -
\frac{1}{2}\right)^2}$.

Since the eigenvalue shape function is $\f(x) = 2x$, the limiting
MSD satisfies
$$
\E{x^2(t)} = \int_0^t \int_0^1 e^{-2 x s} \delta \big(x -
\frac{1}{2}\big) dx ds = 1 - e^{-t}.
$$

\subsubsection{Higher-dimensional diffusions} \label{subsec:high-dim}

The assumption that the particles are diffusing in one-dimensional
space is not essential.  Much like for standard Brownian motion, the
dimension of the diffusion only has an effect on the ACF by a
multiplicative constant.  The exponent of the diffusion is unchanged
and this principle holds for distinguished particle processes.

Consider the family $\xbf(t) = \{\xbf_1(t), \xbf_2(t), \ldots,
\xbf_N(t)\}$, where for $i = {1,\ldots, N}$, the particles have the
form $\xbf_i = \left(x_i^1, x_i^2, \ldots, x_i^d\right)'$ where $d$
is the dimension of the space in which the particles are moving. The
system-wide configuration potential is given by
$$
\Psi(\xbf) = \frac{1}{2} \sum_{n \ne m} \kappa_{nm} |\xbf_n -
\xbf_m|^2.
$$
In this case the interactions decouple in the various components of
the diffusion. The SDE for the $\alpha$-component of the $n$-th bead
is given by
$$
d x_n^\alpha = \sum_{m \ne n} \sum_{\beta = 1}^d \delta_{\alpha
\beta} \kappa_{nm} (x_m^\beta(t) - x_n^\alpha(t))dt + \sigma dW
_n^\alpha(t)
$$
where $\delta_{\alpha \beta}$ is the Kronecker $\delta$-function. We
see that each component conducts its own diffusion independent of
the other components and conclude that, assuming vanishing initial
conditions as usual,
$$
\E{\xbf_n(t)\cdot\xbf_n(s)} =  \frac{\sigma^2 d}{N} \int_0^{t\wedge
s} \sum_{k=1}^N e^{-2 \lambda_k (t+s-2r)} dr
$$
which is simply $d$ times the autocorrelation for one-dimensional
distinguished particles, as in \eqref{eq:xbf-acf}. In contrast to
this observation, the dimension of the underlying connection graph
does have significant effect, which we investigate in Section
\ref{subsec:rouse-higher-d}.

\subsection{The Rouse polymer model}
\label{sec:rouse}

We return to the touchstone example from Section
\ref{subsec:rouse-intro}, the Rouse chain model. We recall that
the connection graph $\gcal_R$ consists of edges $x_n
\leftrightarrow x_{n+1}$ for all $n = \{0, \ldots N-1\}$ as well as
the edge $x_0 \leftrightarrow x_n$.  This yields the system of SDEs
$$
d x_n(t) = \kappa [x_{n-1}(t) - x_n(t)] + \kappa [x_{n+1}(t) -
x_n(t)] + \sigma dW_n(t)
$$
and the diffusive spectrum is given by $\lambda_{k,N} := 4
\sin^2\!\left(\frac{k \pi}{N}\right)$. In terms of Assumption
\ref{as:shape}, the spectral shape function is
$$
\f(x) = 4\sin^2(\pi x)
$$
with $x \in [0,1)$.  Using the Taylor expansion for $\sin(x)$, we
see that this shape function essentially satisfies Assumption
\ref{as:rho} with shape parameter $\rho = 2$. (The only sense in
which it does not satisfy the parameter assumption is that it is not
strictly increasing on the full interval $[0,1)$.  However, this is
easily overcome by restricting to the interval $[0,1/2]$, and
multiplying the MSD by two.)

By Theorem \ref{thm:distinguished}, the family of distinguished
processes $x_{0,N}(t)$ are tight and the limiting process $x(t)$ has
the MSD
\begin{equation} \label{eq:Rouse-msd}
\E{x^2(t)} = 2 \sigma^2 \int_0^t \int_0^{\frac{1}{2}} e^{-4 \kappa
\sin^2(\pi x) s} dx ds
\end{equation}
It follows from the conclusions of Theorem \ref{thm:main} that the
anomalous exponent is $\nu = \frac{1}{2}$.  For an explicit
development of the above Laplace integral, including the order of
the correction terms, Eq.~\ref{eq:Rouse-msd} is happens to be a
worked example in \cite{book-bender-orszag}, Chapter 6.

\subsection{Higher-dimensional Rouse analogues}
\label{subsec:rouse-higher-d}

While we showed in Section \ref{subsec:high-dim} that large-$t$
anomalous exponents do not depend on the dimension of the space in
which the particles reside, we now observe that behavior does change
if the bead-spring network has a higher dimensional connection
graph. We employ a standard technique from graph theory of
constructing complex graphs from simple ones
\cite{2008-mahadevan}.

The Cartesian product of two graphs $\gcal_1$ and $\gcal_2$, which
have vertices $\{v_i\}_{i=1}^N$ and $\{w_i\}_{i=1}^M$, respectively,
consists of vertices enumerated by the set of pairs $\{(v_i,w_j)\}$.
There is an edge $(v_i,w_j) \leftrightarrow (v_k,w_\ell)$ if and
only if either $v_i = v_k$ and $\gcal_2$ contains the edge $w_j
\leftrightarrow w_\ell$, or $w_j = w_\ell$ and $\gcal_1$ contains
the edge $v_i \leftrightarrow v_k$.  When applied to a cycle graph
such as the Rouse graph $\gcal_R$ to itself, the Cartesian product
yields the skeleton of a torus.

The adjacency matrix of a Cartesian product of two graphs is given
by the Kronecker sum of their respective adjacency matrices $A_1
\oplus A_2$.  If $\gcal_1$ is a graph with $N$ vertices and
$\gcal_2$ has $M$ vertices, this sum is defined by $A_1 \oplus A_2
:= A_1 \otimes I_M + I_N \otimes A_2$, where the Kronecker product
$A \otimes B$ is a block matrix whose blocks are of the form $a_{ij}
B$. The resulting Kronecker sum matrix has dimension $NM \times NM$.
The only fact we will use here is that the set of eigenvalues of the
Laplacian matrix associated to the Kronecker sum is given by the set
\cite{2008-mahadevan}
$$
\alpha_{ij} = \{\lambda_i + \mu_j \colon i \in  1 \ldots, N, j \in
1, \ldots M \}.
$$
where $\{\lambda_i\}$ and $\{\mu_j\}$ are the eigenvalues of the Laplace matrices for $\gcal_1$ and $\gcal_2$, respectively.

Denoting the $N$-bead Rouse eigenvalues by $\{\lambda_{j,N}\}$, the MSD of a
distinguished particle in an $N \times N$ system is computed to be
$$
\E{x_N(t)^2} = \frac{1}{N^2} \int_0^t \sum_{i,j} e^{-\kappa
(\lambda_{i,N} + \lambda_{j,N}) s} ds
$$
which converges as $N \to \infty$ to the integral
\begin{equation} \label{eq:rouse-2d-msd}
\E{x(t)^2} = \int_0^t \int_0^1 \int_0^1 e^{-4\kappa (\sin^2(\pi x) +
\sin^2(\pi y)) s} dx dy ds.
\end{equation}
The family of distinguished particles converges in distribution in
the sense of Theorem \ref{thm:distinguished} and it remains
only to perform the Laplace integral asymptotic analysis on
\eqref{eq:rouse-2d-msd}.

As before, we take the leading order behavior near the
$(x,y)$-origin. For large $s$, we have
\begin{align*}
\Phi(s) &:= 4 \int_0^\frac{1}{2} \int_0^\frac{1}{2}
e^{-4 \kappa (\sin^2(\pi x) + \sin^2(\pi y)) s} dx dy \\
&\sim 4 \int_0^\infty \int_0^\infty e^{-4 \kappa \pi^2 (x^2 + y^2)s}
dx dy
\end{align*}
which after a conversion to polar coordinates yields that for $s$
large, $\Phi(s) \sim s^{-1}$. Integrating this asymptotic relation
implies that for large $t$,
$$
\E{x(t)^2} \sim \ln(t).
$$

The Kronecker sum can be iterated arbitrarily many times for higher
dimensional connectivity. Denoting the number of Kronecker sums by
$D$, we see that the associated $\Phi(s)$ satisfies
\begin{align*}
\Phi(s) &\sim 2^D \int_{\rbb_+^D} \exp\left(-4 \kappa \pi^2
s \sum_{i=1}^D x_i^2\right) d \xbf \\
&= C \int_0^\infty e^{-4 \kappa \pi^2 s r^2} r^{D-1} d r \\
&= C s^{-\frac{D}{2}}
\end{align*}
where we have allowed the constant $C = C(D)$ to change from line to
line. After integrating to get the MSD, we see that for $D \geq 3$,
the Rouse model with $D$-dimensional connectivity has MSD that is
bounded for all time. 
%To check to see whether these processes are
%ergodic, we recall from Section \ref{subsec:sou-asymptotic} that it
%suffices to check the integrability of the function
%$$
%\sigma(t,0) = \int_{\rbb^D} \frac{1}{|\xbf|^2} \, e^{-|\xbf|^2 t} d
%\xbf \sim t^{1 - D/2}
%$$
%where we have eliminated some irrelevant constants. For $D = 3,4$,
%this is not integrable and so the process is not ergodic.  However,
%for $D > 5$, the distinguished particle process will be ergodic.

\subsection{Universality of the Rouse exponent}
\label{subsec:rouse-universality}

Returning our focus to the 1-D chain, we address an observation
\cite{book-doi-edwards} that the $t^\frac{1}{2}$ long-term behavior
is universal for a class of models.  Such a class is never precisely
described in the physics literature, but we provide in this section
one characterization.

We construct a network by starting with a Rouse chain where nearest
neighbor edges $x_n \leftrightarrow x_{n+1}$ are weighted by a
single spring constant $\kappa_1 \geq 0$. We generalize the model by
allowing edges of the form $x_n \leftrightarrow x_{n+j}$ which are
respectively given uniform weights $\kappa_j \geq 0$. The weights
are assigned uniformly in the spirit of preserving exchangeability
of the beads. After the appropriate edges are added at the
boundaries (e.g. the edge $x_0 \leftrightarrow x_{N-1}$ with weight
$\kappa_2$) the resulting Laplace matrix is a circulant matrix. For
example, under the assumption that $\kappa_{j} = 0$ for all $j \geq
3$, we have the Laplace matrix:
\begin{equation}
\Lbf = \left(\begin{array}{cccccccc} \kappa_0 & \kappa_{1} &
\kappa_{2} & 0 & \cdots & 0 & \kappa_{-2} & \kappa_{-1} \\
\kappa_{-1} & \kappa_0 & \kappa_{1} & \kappa_{2} & \cdots & 0 &
0 &  \kappa_{-2} \\
\kappa_{-2} & \kappa_{-1} & \kappa_0 & \kappa_1 & \cdots & 0 &
0 &  0 \\
0 & \kappa_{-2} & \kappa_{-1} & \kappa_0 & \cdots & 0 &
0 &  0 \\
0 & 0 & \kappa_{-2} & \kappa_{-1} & \cdots & 0 &
0 &  0 \\
&&\vdots&& \vdots&& & \\
\kappa_{-2} & 0 & 0 & 0 & \cdots & \kappa_1 &
\kappa_0 &  \kappa_{-1} \\
\kappa_{-1} & \kappa_{-2} & 0 & 0 & \cdots & \kappa_2 &
\kappa_1 &  \kappa_0 \\
\end{array} \right)
\end{equation}
where $\kappa_{-j} = \kappa_j$ and $\kappa_0 = -\sum_j
\kappa_j$. By convention the set of indices for the weights
$\{\kappa_j\}$ will be $j \in \{-\lfloor N/2 \rfloor, \ldots, \lceil
N/2 \rceil\}$.

\begin{theorem}[Universality of the Rouse exponent] \label{thm:rouse-universal}
Let $\{\gcal_N\}$ be a family of graphs whose associated Laplacian
matrices are circulant with weights $\{\kappa_j\}_{j \in
\zbb}$ which satisfy $\kappa_{-j} = \kappa_j$ for all $j \geq 1$
and $\kappa_0 = \sum_{j \in \zbb} \kappa_j$. We assume that there exists an integer $K > 0$, such that $\kappa_j = 0$ for all $j > K$. 

Then the family of distinguished particle processes $x_N$ converges in distribution to a
mean zero Gaussian process $x$ with ACF given by \eqref{eq:acf} where the shape function
$\f$ is defined by 
\begin{equation} \label{eq:shape-szego} 
	\f(x) := \sum_{j=-\infty}^\infty e^{2 \pi i x j} \kappa_j 
\end{equation} 
Furthermore MSD the limiting process $x$ satisfies the asymptotic relationships \begin{align*} 
\E{x^2(t)} &\sim t, \qquad t \mbox{ near zero;} \\ \E{x^2(t)} &\sim t^{\frac{1}{2}}, \qquad t \mbox{ large.} 
\end{align*} 
\end{theorem} 

\begin{proof} 
The eigenvalues and eigenvectors of circulant matrices can be computed directly \cite{book-gray}. For any given set of weights
$\{\kappa_j\}$, , the associated circulant matrix has a set of eigenvalue-eigenvector
pairs $\{(\lambda_k, \vbf_k)\}_{k=0}^N$ given by 
\begin{equation*}
	\lambda_k := \sum_{j = -\lfloor N/2 \rfloor}^{\lceil N/2 \rceil} e^{2 \pi i k j/(N+1)} \kappa_j
\end{equation*} and 
\begin{equation*}
	\vbf_k :=
	\frac{1}{\sqrt{N+1}} \left(1, e^{2 k \pi i/(N+1)}, e^{2 k \pi 2 i/(N+1)}, \ldots e^{2 k \pi i N/(N+1)} \right)'.
\end{equation*}
Uniform convergence of the eigenvalues to the shape function \ref{eq:shape-szego} in the sense of Assumption \ref{as:shape} is clear. Tightness of the
associated distinguished particles processes follow as in
Theorem \ref{thm:distinguished}.  The near zero asymptotic
behavior of the MSD is as in Theorem \ref{thm:main}. It remains only
to demonstrate the large-$t$ MSD behavior.

The restrictions that $\kappa_{-j} = \kappa_j$ and $\kappa_0$ is the
sum of all the other weights implies that
\begin{align*}
\f(x) &= \sum_{j=-K}^K e^{2 \pi i x j} \kappa_j = \sum_{j=1}^K 2 \kappa_j \left(1 - \cos(2 \pi x j)\right) = 4 \sum_{j=1}^K \kappa_j \sin^2(\pi x j)
\end{align*}
We immediately see that the Rouse spectrum of Section
\ref{sec:rouse} corresponds to the special case $K = 1$, $\kappa_1 =
1$. Since the sum is finite, applying the Taylor
expansion term-by-term to the series yields
$$
\f(x) \sim 4 \pi^2 x^2 \sum_{j=1}^K j^2 \kappa_j
$$
for $x < 1/K$.  As such, the shape function satisfies Assumption
\ref{as:rho} with shape parameter $\rho = 2$.  The large-$t$ anomalous
exponent $\nu = \frac{1}{2}$ follows immediately.
\end{proof}

Readers familiar with the theory of Toeplitz matrices will recognize
this type of weak convergence for eigenvalues from Szeg\"o's Theorem.
In this more general light, we see that this notion of eigenvalue convergence is more robust than the case stated here, however if one wishes to carry out the program for
non-exchangeable bead-spring systems, the convergence of the eigenvectors will have to be more carefully considered.

\subsection{The inclusion of repulsive forces}
\label{subsec:repulsive}

Seeing this universal nature of the Rouse scaling discourages the
notion that distinguished particle processes will be able to address
the wide range of behaviors seen experimentally.  It is interesting
to note, at least from a mathematical point of view, that if we are
allowed to include repulsive potentials between the beads, a larger
class of exponents becomes available.

For example, suppose that for a generalized Rouse network with
connectivity order 2, we have $\kappa_2 = - \kappa_1/4$.  Then the
leading order term of the expansion for the shape function $\f$
vanishes. One finds that in this case, the spectral parameter is
$\rho = 4$, and the resulting large-$t$ MSD exponent is $\nu = 1 -
\frac{1}{\rho} = \frac{3}{4}$.

In fact, one can create a process with large-$t$ MSD exponent $\nu =
1 - \frac{1}{2n}$ for any $n \in \nbb$, by Fourier inverting the
shape function $\f(x) = \sin^{2n}(x)$.  Such a shape function would
satisfy Assumption \ref{as:rho} with parameter $\rho = 2n$. To be
specific about the coefficients, we let
$$
\kappa_j := \int_0^1 e^{-2 \pi i j x} (e^{\pi i x} - e^{-\pi i
x})^{2n} dx.
$$
for $j \in \{-2n, \ldots 2n\}$. It follows that $\kappa_j =
(-1)^{j+1} \binom{2n}{j}$. The choice to take only even powers of
sine is made to ensure the symmetry $\kappa_j  = \kappa_{-j}$.

\subsection{Rescaled distinguished particle processes and linear
SPDE} \label{subsec:spde}

As a concluding note, we return to our discussion random string model which leads to  another set of \sou processes which have
qualitatively different behavior.

Consider again the Rouse chain model, but now we suppose that as the
number of beads increases, we simultaneously increase the spring
strength by a factor of $N^2$.  Without also increasing the
fluctuation strength by a factor of $\sqrt{N}$, the limiting
structure would collapse to a single point. That is, we define $x_n$
to satisfy the SDE
\begin{equation} \label{eq:defn-rouse-discrete}
d x_n(t) = \kappa N^2 (x_{n+1} - 2 x_n(t) + x_{n-1}(t)) + \sigma
\sqrt{N} dW_n(t)
\end{equation}

In \cite{1983-funaki} the author showed that there is a non-trivial
limiting object under the above rescaling.  We define a family of
functions by
$$
u_N(k/N, t) := x_{k,N}(t)
$$
and take the linear interpolation for the value of $u_N(y,t)$, for
all $y \in (\frac{k}{N},\frac{k+1}{N})$.  Then
$\{u_N(y,t)\}_{N=1}^\infty$ forms a tight family of functions and
the limiting object $u(y,t)$ satisfies the stochastic heat equation
$$
\partial_t u_N(y,t) = \Delta u(y,t) + W(dy,dt)
$$
where $W(dx,dt)$ is a space-time white noise \cite{book-walsh-spde}.  The
second-derivative in space can be anticipated by seeing the
double-difference spring operator in \eqref{eq:defn-rouse-discrete}
as a discrete approximation to the Laplacian $\Delta$.

The limiting distinguished particle process is $u(y,\cdot)$.  The
exact solution can be expressed in the Fourier inversion
$$
u(y,t) = \int_0^t \sum_{k=-\infty}^\infty e^{2 \pi i k y/N} e^{-4 \pi^2 k^2
(t-s)/N^2} \sigma d B_k(s)
$$
where we recognize this is as limit of \sou processes given by
\eqref{eq:sou-defn-x} and \eqref{eq:sou-defn-z} with coefficients
and spectrum
$$
\lambda_{k,N} = \left(\frac{2 \pi k}{N}\right)^2, \qquad c_{k,N} =
\cos\left(\frac{2 \pi k \ell}{N}\right)
$$
One can show that these monomer paths exhibit anomalous diffusion,
$\E{u(y,t)^2} \sim t^{1/2}$ for $t$ near zero rather than for large
$t$.  In fact, the process approaches a stationary distribution for
large times. As mentioned earlier in this section, this process $u(y,\cdot)$ has also received attention recently because its sample paths are locally rougher than Brownian motion, having finite quartic variation \cite{2007-swanson}.

Returning to the discussion in the previous subsection, setting
$\kappa_2 = -\kappa_1 / 4$ results in the system of SDEs
\begin{align*}
d x_n(t) = \kappa (x_{n+1} - 2 x_n(t) + x_{n-1}(t)) + \sigma dW_n(t)
\end{align*}
Presumably by rescaling the spring constants by $N^4$ while
strengthening the noise appropriately, one obtains the stochastic
beam equation in the limit.  One suspects that the local behavior
will be rougher still and have anomalous exponent $\nu = 3/4$.
Similar results should exist for any even number of spatial
derivatives, but we do not pursue this line of thought here.

\section{Acknowledgments}

The author would like to thank Greg Forest, Jonathan Mattingly, Peter March and Davar Khoshnevisan for their feedback and many thought-provoking conversations.  This work was supported in part by NSF DMS Grant No. 0449910.

\bibliographystyle{unsrt}
\bibliography{anomalous-diffusion}

\begin{thebibliography}{10}

\bibitem{2005-suh}
Junghae Suh, Michelle Dawson, and Justin Hanes.
\newblock Real-time multiple-particle tracking: applications to drug and gene
  delivery.
\newblock {\em Advanced Drug Delivery Reviews}, 57:63--78, 2005.

\bibitem{2006-VirtualLung-PNAS}
Hirotoshi Matsui, Victoria~E. Wagner, David~B. Hill, Ute~E. Schwab, Troy~D.
  Rogers, Brian Button, Russell M.~Taylor II, Richard Superfine, Michael
  Rubinstein, Barbara~H. Iglewski, and Richard Boucher.
\newblock A physical linkage between cystic fibrosis airway surface dehydration
  and pseudomonas aeruginosa biofilms.
\newblock {\em Proceedings of the National Academy of the Sciences},
  103(48):18131--18136, 2006.

\bibitem{2007-hanes}
Jung~Soo Suk, Junghae Suh, Samuel~K. Lai, and Justin Hanes.
\newblock Quantifying the intracellular transport of viral and nonviral gene
  vectors in primary neurons.
\newblock {\em Experimental Biology and Medicine}, 232:461--469, 2007.

\bibitem{2002-morgado-anom-diff}
R.~Morgado, F.~A. Oliveira, G.~G. Batrouni, and A.~Hansen.
\newblock Relation between anomalous diffusion and normal diffusion in systems
  with memory.
\newblock {\em Physical Review Letters}, 89(10), 2002.

\bibitem{2004-kupferman-frac-kin}
Raz Kupferman.
\newblock Fractional kinetics in kac–zwanzig heat bath models.
\newblock {\em Journal of Statistical Physics}, 114(1-2), 2004.

\bibitem{2008-vainstein-scaling-law}
Mendeli~H. Vainstein, Luciano~C. Lapas, and Fernando Oliviera.
\newblock Anomalous diffusion.
\newblock {\em Acta Physica Polonica B}, 39(5):1273, 2008.

\bibitem{2008-santamariaholek}
I.~Santamaria-Holek.
\newblock Anomalous diffusion in microrheology: A comparative study.
\newblock 2008.

\bibitem{2008-kou-subdiffusion}
S.~C Kou.
\newblock Stochastic modeling in nanoscale biophysics: Subdiffusion within
  proteins.
\newblock {\em Annals of Applied Statistics}, 2(2):501--535, 2008.

\bibitem{2009-fricks}
John Fricks, Lingxing Yao, Timothy~C. Elston, and M.~Gregory Forest.
\newblock Time-domain methods for diffusive transport in soft matter.
\newblock {\em SIAM Journal of Applied Mathematics}, 69(5):1277--1308, 2009.

\bibitem{2009-jor}
Scott~A. McKinley, Lingxing Yao, and M.~Gregory Forest.
\newblock Transient anomalous diffusion of tracer particles in soft matter.
\newblock {\em Journal of Rheology}, (6), 2009.

\bibitem{book-doi-edwards}
M.~Doi and S.~F. Edwards.
\newblock {\em The Theory of Polymer Physics}.
\newblock Oxford University Press, 1986.

\bibitem{1990-kremer-grest}
Kurt Kremer and Gary~S. Grest.
\newblock Dynamics of entangled linear polymer melts: A molecular-dynamics
  simulation.
\newblock {\em J.~Chem.~Phys.}, 92(8):5057--5086, 1990.

\bibitem{book-rubinstein-colby}
Michael Rubinstein and Ralph~H. Colby.
\newblock {\em Polymer Physics}.
\newblock Oxford University Press, 2003.

\bibitem{2002-kupferman-stuart}
R.~Kupferman, A.~M. Stuart, J.~R. Terry, and P.~F. Tupper.
\newblock Long-term behaviour of large mechanical systems with random initial
  data.
\newblock {\em Stochastics \& Dynamics}, 2(4):533 --, 2002.

\bibitem{2004-kupferman-fitting-sde}
R.~Kupferman and A.M. Stuart.
\newblock Fitting sde models to nonlinear kac-zwanzig heat bath models.
\newblock {\em Physica D: Nonlinear Phenomena}, 199(3-4):279 -- 316, 2004.

\bibitem{2001-zwanzig-book}
R.~Zwanzig.
\newblock {\em Non-equilibrium Statistical Mechanics}.
\newblock Oxford University Press, 2001.

\bibitem{1956-zimm}
Bruno~H. Zimm.
\newblock Dynamics of polymer molecules in dilute solution: viscoelasticity,
  flow birefringence and dielectric loss.
\newblock {\em The Journal of Chemical Physics}, 24(2):269--278, 1956.

\bibitem{1989-ottinger-rabin}
Hans~Christian \"Ottinger and Yitzak Rabin.
\newblock Diffusion equation versus coupled langevin equations approach to
  hydrodynamics of dilute polymer solutions.
\newblock {\em Journal of Rheology}, 33:725--743, 1989.

\bibitem{book-bender-orszag}
C.~M. Bender and S.~A. Orszag.
\newblock {\em Advanced Mathematical Methods for Scientists and Engineers I:
  Asymptotic Methods and Perturbation Theory}.
\newblock McGraw-Hill, 1978.

\bibitem{book-revuz-yor}
Daniel Revuz and Marc Yor.
\newblock {\em Continuous Martingales and Brownian Motion}.
\newblock Springer-Verlag Berlin Heidlberg, 1991.

\bibitem{1983-funaki}
Tadhisa Funaki.
\newblock Random motions of strings and related stochastic evolution equations.
\newblock {\em Nagoya Math Journal}, 89:129--183, 1983.

\bibitem{2007-swanson}
Jason Swanson.
\newblock Variations of the solution to a stochastic heat equation.
\newblock {\em Annals of Probability}, 35(6):2122--2159, 2007.

\bibitem{book-chung}
Fan Chung.
\newblock {\em Spectral Graph Theory}.
\newblock American Mathematical Society, 1997.

\bibitem{2008-mahadevan}
Sridhar Mahadevan.
\newblock {\em Representation Discovery using Harmonic Analysis}.
\newblock Morgan and Claypool Publishers, 2008.

\bibitem{book-gray}
Robert~M. Gray.
\newblock {\em Toeplitz and Circulant Matrices: A Review}.
\newblock Foundations and Trends in Communications and Information. now
  Publishers Inc., 2006.

\bibitem{book-walsh-spde}
John~B. Walsh.
\newblock An introduction to stochastic partial differential equations.
\newblock In {\em \'{E}cole d'\'et\'e de probabilit\'es de {S}aint-{F}lour,
  {XIV}---1984}, volume 1180 of {\em Lecture Notes in Math.}, pages 265--439.
  Springer, Berlin, 1986.

\end{thebibliography}

\end{document}